\documentclass[a4paper,12pt]{amsart}
\pdfoutput=1
\usepackage{amsmath,amssymb,amsfonts,amsthm}
\usepackage{setspace}
\usepackage{graphicx}

\newtheorem{theorem}{Theorem}
\newtheorem*{theorem*}{Theorem}
\newtheorem{lemma}{Lemma}

\newtheorem{remark}{Remark}
\newtheorem*{remark*}{Remark}
\newtheorem{definition}{Definition}
\newtheorem*{definition*}{Definition}

\DeclareMathOperator{\rk}{rk}

\begin{document}
\onehalfspacing

\title[Commutator width of Chevalley groups]{Commutator width of Chevalley groups\\over rings of stable rank $1$}
\keywords{commutator width, unitriangular factorization, Gauss decomposition}
\author{Andrei Smolensky}
\email{andrei.smolensky@gmail.com}
\address{Department of Mathematics and Mechanics, Saint Petersburg State University, Universitetsky prospekt, 28, 198504, Peterhof, Saint Petersburg, Russia}
\thanks{Research was supported by RFFI (grants 12-01-00947-a and 14-01-00820) and by State Financed task project 6.38.191.2014 at Saint Petersburg State University.}
\date{\today}
\subjclass[2000]{Primary 20G07; Secondary 20G41}
\begin{abstract}
An estimate on the commutator width is given for Chevalley groups over rings of stable rank $1$, and the general method suitable for other rings of small dimension.
\end{abstract}
\maketitle

\section{Introduction}
The study of commutators in linear groups over fields has a rich history, culminating in the celebrated proof of Ore conjecture \cite{EllGorOreConj,LieOBrShaTieOre}, while the groups over rings received much less attention. It was shown in \cite{VasWheCommComp}, that for an associative ring $R$ of stable rank $1$ the group $GL(n,R)$, $n\geqslant3$ has the commutator width $\leqslant2$, i.e. that every element of its commutator subgroup can be written as a product of two commutators. It was then generalised (with somewhat worse bounds) in \cite{ArlVasYouCommUnitary} to symplectic, orthogonal and unitary groups in even dimension in the context of hyperbolic unitary groups \cite{BakVavUnitary}. The goal of the present paper is to provide a similar result for exceptional groups.

The proof follows the line of those in \cite{VasWheCommComp,ArlVasYouCommUnitary}, but tries to avoid explicit matrix calculations, thus giving a simpler and (almost) uniform treatment for Chevalley groups of all normal types. Apart from exceptional groups, it also covers Spin and odd-dimensional orthogonal groups, which were not considered in the previous papers.

It was a surprise for the author that the case of special linear group over rings of stable rank $1$ is not presented in \cite{VasWheCommComp} or anywhere else. Apparently, the reason is that one has to do some additional considerations as in Lemma~\ref{lemma:inverse-of-companion} and Remarks~\ref{rmk:inverse-of-companion-a4k+1} and \ref{rmk:conjugate-to-companion-minus-a4k+1} (compare with Proposition~8 of \cite{VasWheCommComp}), leading to certain (insignificant) complications in the proof of Theorem~\ref{thm:comm-width-sr1} for $SL_{4k+2}$.

Let $\Phi$ be a reduced irreducible root system, $\alpha_1,\ldots,\alpha_\ell$ its fundamental roots (numbered as in Bourbaki), $W(\Phi)$ the corresponding Weyl group, generated by the simple reflections $\sigma_1,\ldots,\sigma_\ell$. For a commutative ring $R$ with $1$ by the Chevalley group $G(\Phi,R)$ we mean the group of point of the correspong Chevalley-Demazure group scheme $G(\Phi,-)$. Unless specified otherwise, all groups are assumed to be simply connected.

We extensively use the weight diagrams, see \cite{PloSemVavVBRA,VavThirdLook,VavDIY}. Weight diagrams allow to visualize the action of the elementary root unipotents $x_\alpha(t)$ (the generators of the elementary subgroup $E(\Phi,R)$). Below is the weight diagram for the natural vector representation of $SL_{\ell+1}$, the nodes correspond to the weights of the representation, and the edges to the fundamental roots. $x_{ij}(t)$ acts on $(v_k)_{k=1}^{\ell+1}$ by adding $tv_j$ to $v_i$, or, in terms of the diagram, along the chain connecting $j$ to $i$.
\begin{figure}[h]
\includegraphics{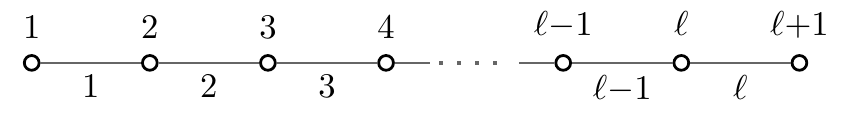}
\caption{$(\mathsf{A}_\ell,\varpi_1)$}
\label{fig:al-nat}
\end{figure}

We call a set of roots $S\subseteq\Phi$ closed if for any $\alpha,\beta\in S$, if $\alpha+\beta\in\Phi$ is a root, then $\alpha+\beta\in S$. Two important examples of closed sets of roots are the following. Let $m_i(\alpha)$, $i=1,\ldots,\ell$ denote the coefficients in the expansion of $\alpha$ as an integer linear combination of the fundamental roots, then put
\[ \Sigma_k=\{\alpha\in\Phi\mid m_k(\alpha)\geqslant1 \},\quad
\Delta_k=\{\alpha\in\Phi\mid m_l(\alpha)=0\}. \]
The sets $\Sigma_k$ are unipotent (i.e. $S\cap-S=\varnothing$), and $\Delta_k$ are symmetric (i.e. $S=-S$). The notation $\Sigma_k^{=n}$, $\Sigma_k^{\leqslant n}$ and $\Sigma_k^{\geqslant n}$ is self-explaining.

The sets of all positive and negative roots $\Phi^+,\Phi^-$ are also both closed and unipotent.

To a closed set of roots $S$ we associate a subgroup $E(S,R)=\langle x_\alpha(t)\mid\alpha\in S,\ t\in R\rangle$. The ring $R$ is often clear from the context and thus omitted in the notation. If $S$ is unipotent, we sometimes write $U(S)$ instead of $E(S)$. The unitriangular subgroups $U(\Phi^\pm)$ are denoted by $U^\pm$.

$U(\Sigma_k)$ is the unipotent radical of the corresponding parabolic subgroup $P_k$ (or $E(\Delta_k\cup\Sigma_k)$). Levi decomposition states that $E(\Delta_k\cup\Sigma_k)$ is the semi-direct product of its elementary Levi subgroup $E(\Delta_k)$ and its normal subgroup $U(\Sigma_k)$.

\section{Commutators and companion matrices}
\begin{definition}\label{def:weyl-trajectories}
Fix an element $w\in W(\Phi)$ and a natural number $n$ and set
\begin{align*}
&\Omega^w_n=\{\alpha\in\Phi^+\mid w^{n+1}\alpha\in\Phi^-,\ w\alpha,w^2\alpha,\ldots,w^n\alpha\in\Phi^+\},\\
&\Theta^w=\{\alpha\in\Phi^+\mid w^k\alpha\in\Phi^+\ \forall k\in\mathbb{Z}\}.
\end{align*}
When the choice of particular element $w$ is clear from the context, we usually omit the super index and simply write $\Theta$ and $\Omega_n$.
\end{definition}
Note that $\Phi^+=\Theta\cup\bigcup_{k\geqslant0}\Omega_k$ and the union is disjoint.
\begin{remark}
$\Theta$ is closed for any $w\in W$.
\end{remark}
\begin{proof}
Suppose there are $\alpha,\beta\in\Theta$ with $\alpha+\beta\in\Phi^+\setminus\Theta$. Then there is some $k\geqslant0$ such that $w^k(\alpha+\beta)\in\Phi^-$. So either $w^k\alpha$ or $w^k\beta$ is negative.
\end{proof}
\begin{lemma}
$\Theta\cup\bigcup_{k=0}^n\Omega_k$ is closed for any $n$. As a corollary, $\Phi^+\setminus\cup_{k\geqslant n}\Omega_k$ is closed for any $n$.
\end{lemma}
\begin{proof}
Suppose there are $\alpha,\beta\in\Theta\cup\bigcup_{k=0}^n\Omega_k$ such that $\alpha+\beta\in\cup_{k>n}\Omega_k$. Then there exists $m>n$ with $w^{m+1}(\alpha+\beta)\in\Phi^-$ and $w^i(\alpha+\beta)\in\Phi^+$ for all $i=0,\ldots,m$.

Thus either $w^{m+1}\alpha$ or $w^{m+1}\beta$ is negative, so one of $\alpha,\beta$ lies outside of $\Theta$. Since $w^{n+1}(\alpha+\beta)\in\Phi^+$, one has $\alpha\in\Theta\cup\bigcup_{k>n}\Omega_k$ or $\beta\in\Theta\cup\bigcup_{k>n}\Omega_k$. If $\alpha\in\Theta$, then $\beta\in\cup_{k>n}\Omega_k$, a contradiction, and similarly for $\beta\in\Theta$.
\end{proof}

We write, as usual
\[ w_\alpha(u)=x_\alpha(u)x_{-\alpha}(-u^{-1})x_\alpha(u),\quad h_\alpha(u)=w_\alpha(u)w_\alpha(-1). \]
\begin{remark}\label{rmk:R3}
For any $\alpha,\beta\in\Phi$, $t\in R$, $u,v\in R^*$
\begin{align*}
& w_\alpha(u)x_\beta(t)w_\alpha(u)^{-1}=x_{\sigma_\alpha\beta}\left(\pm u^{-\langle\beta,\alpha\rangle}t\right),\\
& w_\alpha(u)w_\beta(v)w_\alpha(u)^{-1}=w_{\sigma_\alpha\beta}\left(\pm u^{-\langle\beta,\alpha\rangle}v\right),\\
& h_\alpha(u)w_\beta(v)h_\alpha(u)^{-1}=w_\beta\left(u^{\langle\beta,\alpha\rangle}v\right).
\end{align*}
\end{remark}
The above relations hold on the level of Steinberg group, while on the level of elementary group they imply the following formula (note that signs cancel out):
\[ w_\alpha(1)h_\beta(\varepsilon)w_\alpha(1)^{-1}=h_{\sigma_\alpha\beta}(\pm\varepsilon)h_{\sigma_\alpha\beta}(\pm1)^{-1}=h_{\sigma_\alpha\beta}(\varepsilon). \]

The extended Weyl group $\widetilde{W}(\Phi)$ is the subgroup of $G(\Phi,R)$, generated by $w_\alpha(1)$, $\alpha\in\Phi$. If $2\neq0$ in $R$, it coincides with the $N(\Phi,\mathbb{Z})$, the group of integer point of the torus normalizer. It is an extension $C_2^\ell\hookrightarrow N(\Phi,\mathbb{Z})\twoheadrightarrow W(\Phi)$, and the action of the generators on the kernel is described by the above formula.

\begin{definition}\label{def:coxeter-lift}
Denote by $\widetilde\pi$ the following element of the Weyl group:
\begin{itemize}
\item $\mathsf{A}_\ell,\mathsf{B}_\ell,\mathsf{C}_\ell$: $\widetilde\pi=\sigma_1\sigma_2\ldots \sigma_\ell$, a Coxeter element;
\item $\mathsf{D}_\ell$: $\widetilde\pi=\sigma_\ell\ldots \sigma_2\sigma_1$;
\item $\mathsf{E}_6$: $\widetilde\pi=\sigma_1\sigma_3\sigma_4\sigma_5\sigma_6$, a Coxeter element of an $\mathsf{A}_5$ subsystem;
\item $\mathsf{E}_7$: $\widetilde\pi=\sigma_1\sigma_3\sigma_2\sigma_4\sigma_5\sigma_6\sigma_7$;
\item $\mathsf{E}_8$: $\widetilde\pi=\sigma_1\sigma_3\sigma_2\sigma_4\sigma_5\sigma_6\sigma_7\sigma_8$;
\item $\mathsf{F}_4$: $\widetilde\pi=\sigma_1\sigma_2\sigma_3\sigma_4$;
\item $\mathsf{G}_2$: $\widetilde\pi=\sigma_2\sigma_1$.
\end{itemize}
By $\pi$ denote a lift of $\widetilde\pi$ to the extended Weyl group, obtained by sending $\sigma_i$ to $w_i(1)$.
\end{definition}
For any Coxeter element $w_c$ one has $\Theta^{w_c}=\varnothing$. Thus with the choice of $\widetilde{\pi}$ as above $\Theta^{\widetilde{\pi}}=\varnothing$ in all cases except $\mathsf{E}_6$, when $\Theta^{\widetilde{\pi}}=\Sigma_2$.
\begin{definition}
A companion matrix is an element of the form $u\pi$ with $u\in U(\Sigma)$, where $\Sigma=\Omega^{\widetilde{\pi}}_0$ for $\Phi\neq\mathsf{E}_6$ and $\Sigma=\left(\Theta^{\widetilde{\pi}}\setminus\{\alpha_2\}\right)\cup\Omega^{\widetilde{\pi}}_0$ for $\Phi=\mathsf{E}_6$. Depending on the root system it can be described as:
\begin{itemize}
\item $\mathsf{A}_\ell$: $\Sigma=\Sigma_\ell$;
\item $\mathsf{B}_\ell$: $\Sigma=\left(\Sigma_\ell^{=2}\cap\Sigma_{\ell-1}^{=1}\right)\cup\{\alpha_\ell\}$ (marked black on the weight diagram of the adjoint representation, see Figure~\ref{fig:bl-adj});
\item $\mathsf{C}_\ell$: $\Sigma=\left(\Sigma_\ell^{=1}\cap\Sigma_{\ell-1}^{=1}\right)\cup\{\alpha_\ell\}=\Sigma_\ell^{=1}\cap\Sigma_{\ell-1}^{\leqslant1}$ (Figure~\ref{fig:cl-adj});
\item $\mathsf{D}_\ell$: $\Sigma=\Sigma_1\cap\left(\Delta_\ell\cup\Delta_{\ell-1}\right)$ (Figure~\ref{fig:dl-adj});
\item $\mathsf{E}_6$: $\Sigma=\left(\Sigma_6\cap\Delta_2\right)\cup\left(\Sigma_2\setminus\{\alpha_2\}\right)$ (Figure~\ref{fig:e6-adj});
\item $\mathsf{E}_7,\mathsf{E}_8,\mathsf{F}_4,\mathsf{G}_2$: see Figures~\ref{fig:e7-adj},~\ref{fig:e8-adj},~\ref{fig:f4-adj},~\ref{fig:g2-adj}.
\end{itemize}
\end{definition}
The above description (for $\Phi\neq\mathsf{E}_6$) is obtained as follows: first, one checks that $\widetilde{\pi}$ sends the right hand side to $\Phi^-$ and that the number of roots in it equals the rank of $\Phi$. Then it remains to note that $\left|\Omega^{\widetilde{\pi}}_0\right|=\rk(\Phi)$. This follows from the fact that all orbits of a Coxeter element $w_c$ have the same size, equal to the Coxeter number $h$ (since $w_c$ acts by rotation by $2\pi/h$ on its Coxeter plane and no root projects to zero) and that $\left|\Phi\right|=h\cdot\rk(\Phi)$. For $\Phi=\mathsf{E}_6$ one applies this argument to the $\mathsf{A}_5$-subsystem $\Delta_2$.

\begin{lemma}\label{lemma:conjugate-to-companion}
For any $u\in U^+(\Phi)$ exists $\eta\in U^+(\Phi)$ such that $\eta u\pi\eta^{-1}$ is a companion matrix.
\end{lemma}
\begin{proof}
Consider $\Omega_k$ for $w=\widetilde{\pi}$ (see Definition~\ref{def:weyl-trajectories}) and denote by $N$ the maximal natural number such that $\Omega_N\neq\varnothing$.

Write $u$ as a product $\theta v$, where $\theta=\prod_{\alpha\in\Omega_N}x_\alpha(c_\alpha)$ and $v\in E(\Phi^+\setminus\Omega_N)$. Consider the conjugate $\theta^{-1}u\pi\theta$. It follows from Remark~\ref{rmk:R3} that for $\alpha\in\Omega_N$ $\pi x_\alpha(c_\alpha)\pi^{-1}\in E(\Omega_{N-1})\subset E(\Phi^+\setminus\Omega_N)$, and thus $\pi\theta\in E(\Phi^+\setminus\Omega_N)\pi$, so $\theta^{-1}u\pi\theta=u'\pi$ for some $u'\in E(\Phi^+\setminus\Omega_N)$, since the latter set of roots is closed.

Now we can rewrite $u'$ as a product $\theta'v'$, where $\theta'=\prod_{\alpha\in\Omega_{N-1}}x_\alpha(c_\alpha)$ and $v'\in E(\Phi^+\setminus(\Omega_N\cup\Omega_{N-1}))$. Repeat the previous step to get an element of the form $u''\pi$ with $u''\in E(\Phi^+\setminus(\Omega_N\cup\Omega_{N-1}))$.

Repeating this procedure $N-1$ times, we eventually get an element of the form $u\pi$ with $u\in E(\Phi^+\setminus\cup_{k>0}\Omega_k)=E(\Theta\cup\Omega_0)$.

Since $\Theta\cup\Omega_0$ coincides with $\Sigma$ in all cases except $\mathsf{E}_6$, we are almost done. For $\Phi=\mathsf{E}_6$ one additionally has to eliminate $x_{\alpha_2}(*)$ in the same way (which gives the closed set $\Sigma$).
\end{proof}
A lift $w$ of a Coxeter element to the extended Weyl group is essentially a set of signs $s_i^w=\pm1$, assigned to the fundamental roots $\alpha_i$.
\begin{lemma}\label{lemma:coxeter-lifts-conjugate}
Let $\Phi\neq\mathsf{A}_\ell,\mathsf{D}_\ell,\mathsf{E}_7$ and $w_1,w_2$ be two lifts of a single Coxeter element to the extended Weyl group $N(\Phi,\mathbb{Z})$. Then $w_1$ and $w_2$ are conjugated under the action of $H(\Phi,\mathbb{Z})$.
\end{lemma}
\begin{proof}
 The action of an element $h\in H(\mathbb{Z})$ changes some of $s_i^w$, while leaving the others intact. It follows from Remark~\ref{rmk:R3} that $h_\alpha(-1)$ only changes signs assigned to the roots $\beta$ with odd $\langle\beta,\alpha\rangle$. The latter can be easily computed, for example, as follows: if $\beta-r\alpha,\ldots,\beta+q\alpha$ is the $\alpha$-series through $\beta$, then $\langle\beta,\alpha\rangle=r-q$.
 
 To find a suitable element of $H(\mathbb{Z})$ we have to do some case-by-case analysis. In each case we provide a procedure for transforming one set of signs ($w_1$) into another ($w_2$), that is a chain of ``elementary'' transformations $w\mapsto {}^{h_\alpha(\pm1)}w$. The signs of the current value of $w$ will be denoted simply by $s_i$.
 
 $\Phi=\mathsf{B}_\ell$, $\ell\geqslant 3$: we start with obtaining the desired value of $s_\ell$ by conjugating $w$ with $h_{\alpha_{\ell-1}}(\pm1)$. Now we take $\gamma_1=\alpha_1+\alpha_2+\ldots+\alpha_\ell$ and note that $\langle\alpha_1,\gamma_1\rangle=1$ and $\langle\alpha_\ell,\gamma_1\rangle=0$. This shows that conjugating with $h_{\gamma_1}$ allows us to change $s_1$ while not changing $s_\ell$.  Analogously, we can take $\gamma_k=\alpha_{k-1}+2\alpha_k+\ldots+2\alpha_\ell$ to change $s_k$, $k=\ell-1,\ell-2,\ldots,3$. Each of $h_{\gamma_k}$ affects only $s_k$ and $s_{k-1}$, and the latter is fixed by $h_{\gamma_{k-1}}$. The last step is to take $\gamma_2=\alpha_{\max}=\alpha_1+2\alpha_2+\ldots+2\alpha_\ell$, for $h_{\gamma_2}$ changes only $s_2$.
 
 $\Phi=\mathsf{C}_\ell$, $\ell\geqslant2$: we start with fixing $s_\ell$ by conjugating $w$ with $h_\gamma$ for $\gamma=\alpha_{\ell-1}+\alpha_\ell$ (indeed, $\langle\alpha_\ell,\gamma\rangle=1$). Now we change $s_k$, $k=1,2,\ldots,\ell-1$ by using $h_{\alpha_{k+1}}$.
 
 $\Phi=\mathsf{E}_6$: first change $s_2$ by conjugating with $h_{\alpha_4}$, then use $\alpha_1$ and $\alpha_6$ for $s_3$ and $s_5$, then $\alpha_3$ and $\alpha_5$ for $s_1$ and $s_6$, and finally $\alpha_2$ to change $s_4$.
 
 $\Phi=\mathsf{E}_8$: use $\alpha_3$ to change $s_1$, then use $\alpha_4$, $\alpha_5$, $\alpha_6$, $\alpha_7$, $\alpha_8$ to change $s_2$, $s_4$, $s_5$, $s_6$, $s_7$ and finish by using $\alpha_1$ for $s_3$ and $\alpha_{\max}$ for $s_8$.
 
 $\Phi=\mathsf{F}_4$: use $\alpha_1$ and $\alpha_2$ for $s_2$ and $s_1$, then $\alpha_3$ and $\alpha_4$ for $s_4$ and $s_3$.
 
 $\Phi=\mathsf{G}_2$: use $\alpha_1$ and $\alpha_2$ for $s_2$ and $s_1$.
\end{proof}
The failure of the above lemma for $\mathsf{E}_7$ is amusing, yet the reason is unclear.

Since for $\Phi=\mathsf{A}_\ell,\mathsf{D}_\ell,\mathsf{E}_7$ we can't use Lemma~\ref{lemma:coxeter-lifts-conjugate}, we have to do some additional calculations in these cases.

Let $w_0$ denote the longest element of the Weyl group. We write $\widehat{w_0}$ for its obvious lift to the extended Weyl group, obtained by sending each of $\sigma_i$ in the reduced expression to $w_i(1)$. We will later fix another lift in case $\Phi=\mathsf{A}_\ell$.

\begin{lemma}\label{lemma:coxeter-lifts-nice}
For $\Phi=\mathsf{A}_\ell,\mathsf{D}_\ell,\mathsf{E}_7$ one has $\widehat{w_0}w_i(1)\widehat{w_0}^{-1}=w_j(1)$, where $\alpha_j=-w_0(\alpha_i)$.
\end{lemma}
\begin{proof}
For $\mathsf{A}_\ell$ and $\mathsf{D}_\ell$ this can be done by explicit matrix calculation (this is done for orthogonal group in \cite{ArlVasYouCommUnitary} and immediately follows for Spin group, since the central factor doesn't play any role).

For $E_7$ it is not a good idea to write down matrices, but one can do something very similar. Namely, after choosing a positive basis for the microweight representation $(\mathsf{E}_7,\varpi_7)$ one has very simple and explicit description of the action of $\widetilde{W}(\Phi)$.

Let $\Lambda$ denote the set of weights, $\alpha$ a fundamental root, then
\[ w_\alpha(1)v^\lambda=
\begin{cases}
v^\lambda, & \text{if }\lambda\pm\alpha\notin\Lambda,\\
v^{\lambda+\alpha}, & \text{if }\lambda+\alpha\in\Lambda,\\
-v^{\lambda-\alpha}, & \text{if }\lambda-\alpha\in\Lambda.
\end{cases} \]
Now elements of $\widetilde{W}(\Phi)$ act by signed permutation on $\Lambda$, so it is a routine to check that $\widehat{w_0}w_i(1)=w_i(1)\widehat{w_0}$ for each $i$. It is, however, much less amusing, so the author advises to put it into a computer instead.
\end{proof}
As a corollary, we see that in case $\Phi=\mathsf{D}_\ell$ one has $\widehat{w_0}\pi\widehat{w_0}=\pi$, since $w_\ell(1)$ and $w_{\ell-1}(1)$ commute.

Let $p_n$ denote the $n\times n$ peridentity matrix, that is $p_n=(\delta_{i,n-j})_{i,j=1,\ldots,n}$. If $n\neq4k+2$, then either $\det(p_n)=1$ or $\det(-p_n)=1$, so for $SL(n,R)=G(\mathsf{A}_{n-1},R)$ we specifically fix $\widehat{w_0}$ to be $p_n$ or $-p_n$, which is not the obvious lift of $w_0$. But with this choice one has
\[ \widehat{w_0}\cdot w_\ell(-1)\ldots w_1(-1)\cdot\widehat{w_0}^{-1}=w_1(1)\ldots w_\ell(1), \]
which can be written simply as $\widehat{w_0}\pi^{-1}\widehat{w_0}^{-1}=\pi$.

For $n=4k+2$ neither $p_n$ nor $-p_n$ lies in $SL(n,R)$. In this case we set $\widehat{w_0}$ to be an anti-diagonal matrix with $1$'s and $-1$'s alternating. Then $\widehat{w_0}\pi^{-1}\widehat{w_0}^{-1}=-\pi$.

\begin{lemma}\label{lemma:inverse-of-companion}
Let $\Phi=\mathsf{A}_\ell$, $\ell\neq4k+1$ or $\Phi=\mathsf{E}_6$. If $x$ is similar to a companion matrix, then so is $x^{-1}$.
\end{lemma}
\begin{proof}
Write $\eta x\eta^{-1}=u\pi$ for some $u\in U(\Sigma)$, so $\eta x^{-1}\eta^{-1}=\pi^{-1}u^{-1}$.

In case $\Phi=\mathsf{A}_\ell$, $\ell\neq4k+1$ conjugate it with $\widehat{w_0}$ to get $\pi u'$, where $u'\in U(w_0\Sigma)$. Note that $w_0\Sigma_\ell=-\Sigma_1=\widetilde{\pi}\Sigma_\ell$, thus $\pi u'=u''\pi$ for some $u''\in U(\Sigma_\ell)$.

In case $\Phi=\mathsf{E}_6$ take $v_0$ to be the longest element of $\Delta_2$ and $\widehat{v_0}$ its lift to the torus normalizer. Conjugate $\pi^{-1}u^{-1}$ with $\widehat{v_0}$ to get $\rho u'$, where $u'\in U(v_0\Sigma)$ and $\rho$ is a (probably different) lift of $\widetilde{\pi}$, and with $\rho^{-1}$ to get $u'\rho$. Conjugating it, if necessary, with a suitable element of $H(\mathbb{Z})$ (see Lemma~\ref{lemma:coxeter-lifts-conjugate}), we can assume $\rho=\pi$. Since $v_0\left(\Sigma_6\cap\Delta_2\right)=\widetilde{\pi}\left(\Sigma_6\cap\Delta_2\right)=-\Sigma_1\cap\Delta_2$ and $v_0\Sigma_2=\widetilde{\pi}\Sigma_2=\Sigma_2$, one has $u'\pi=\pi u''$ for $u''\in U(\Sigma\cup\{\alpha_2\})$, which is conjugated with $u''\pi$. It remains to conjugate it with $x_{\alpha_2}(*)$ as in the proof of Lemma~\ref{lemma:conjugate-to-companion}.
\end{proof}

\begin{remark}\label{rmk:inverse-of-companion-a4k+1}
Let $\Phi=\mathsf{A}_{4k+1}$. If $x$ is similar to a companion matrix, then $x^{-1}$ is similar to a minus companion matrix.
\end{remark}
\begin{proof}
Repeat the proof of Lemma~\ref{lemma:inverse-of-companion}, now using $\widehat{w_0}\pi^{-1}\widehat{w_0}^{-1}=-\pi$.
\end{proof}

\begin{lemma}\label{lemma:conjugate-to-companion-minus}
Let $\Phi\neq\mathsf{A}_{4k+1}$. For any $v\in U^-(\Phi)$ exists $\eta\in E(\Phi)$ such that $\eta v\pi\eta^{-1}$ is a companion matrix.
\end{lemma}
\begin{proof}
Note that $\widehat{w_0}v\widehat{w_0}^{-1}\in U^+(\Phi)$.

In case $\Phi=\mathsf{B}_\ell,\mathsf{C}_\ell,\mathsf{D}_\ell,\mathsf{E}_7,\mathsf{E}_8,\mathsf{F}_4,\mathsf{G}_2$ one has $w_0\widetilde{\pi}w_0=\widetilde{\pi}$ and thus $\widehat{w_0}\pi \widehat{w_0}^{-1}=\rho$ for some lift $\rho$. This lift is either equal to $\pi$ (in cases $\mathsf{D}_\ell,\mathsf{E}_7$ by Lemma~\ref{lemma:coxeter-lifts-nice}) or can be transformed to $\pi$ by conjugating with an element of $H(\mathbb{Z})$ (in all other cases by Lemma~\ref{lemma:coxeter-lifts-conjugate}). Then one applies Lemma \ref{lemma:conjugate-to-companion}.

If $\Phi=\mathsf{A}_\ell$, $\ell\neq4k+1$ or $\Phi=\mathsf{E}_6$, the longest element sends $\widetilde{\pi}$ to its inverse, so $\widehat{w_0}\pi\widehat{w_0}^{-1}=\rho^{-1}$. Thus $\widehat{w_0}v\pi\widehat{w_0}^{-1}\in U^+(\Phi)\rho^{-1}$ and
\[ \rho^{-1}\widehat{w_0}v\pi\widehat{w_0}^{-1}\rho\in\rho^{-1}U^+(\Phi), \]
which is by Lemma~\ref{lemma:inverse-of-companion} similar to a companion matrix as the inverse of an element from $U^+(\Phi)\rho$ (in case $\Phi=\mathsf{A}_\ell$ one can assume $\rho=\pi$ as in the proof of Lemma~\ref{lemma:inverse-of-companion}, while in case $\Phi=\mathsf{E}_6$ one applies Lemma~\ref{lemma:coxeter-lifts-conjugate}).
\end{proof}

\begin{remark}\label{rmk:conjugate-to-companion-minus-a4k+1}
Let $\Phi=\mathsf{A}_{4k+1}$. For any $v\in U^-(\Phi)$ exists $\eta\in E(\Phi)$ such that $\eta v\pi\eta^{-1}$ is a minus companion matrix.
\end{remark}
\begin{proof}
Repeat the proof of Lemma~\ref{lemma:conjugate-to-companion-minus}, using Remark~\ref{rmk:inverse-of-companion-a4k+1} instead of Lemma~\ref{lemma:inverse-of-companion}.
\end{proof}

\section{Proof of the main result}
\begin{lemma}\label{lemma:g-minus-one}
For the vector representation $(\mathsf{A}_\ell,\varpi_1)$ there exists an element $g\in E(\ell+1,R)$, such that $g-1\in E(\ell+1,R)$.
\end{lemma}
\begin{proof}
For $\ell=1$ and $\ell=2$ such elements are delivered by the matrices
\[
g_1=\begin{pmatrix}1&-1\\1&0\end{pmatrix},\quad
g_2=\begin{pmatrix}1&0&1\\1&1&0\\0&1&0\end{pmatrix},
\]
and for arbitrary $\ell$ one simply composes them into a block diagonal matrix.
\end{proof}
\begin{lemma}\label{lemma:companion-as-commutators}
If $\rk\Phi\geqslant2$ and $\Phi\neq\mathsf{C}_\ell$, then every element $\theta\in U(\Sigma)$ is a product of at most $N$ commutators, where
\begin{itemize}
\item $N=1$ in case $\Phi=\mathsf{A}_\ell,\mathsf{F}_4,\mathsf{G}_2$;
\item $N=2$ in case $\Phi=\mathsf{B}_\ell,\mathsf{C}_\ell,\mathsf{D}_\ell,\mathsf{E}_7,\mathsf{E}_8$;
\item $N=3$ in case $\Phi=\mathsf{E}_6$.
\end{itemize}
\end{lemma}
\begin{proof}
We start with working out the case $\Phi=\mathsf{A}_\ell$. Denote $\Delta=\Delta_\ell$, then the Levi factor $E(\Delta)$ acts on the unipotent radical $U(\Sigma_\ell)$.

Write $\theta=\prod_{\alpha\in\Sigma}x_\alpha(\xi_\alpha)$. An element $g\in E(\Delta)=E(\ell,R)$ acts on the vector consisting of $\xi_\alpha$ exactly as in $(\mathsf{A}_{\ell-1},\varpi_1)$. To avoid confusion the result will be denoted by ${}^g(\xi_\alpha)$.

Let $g\in E(\Delta)$ be the element, constructed in Lemma \ref{lemma:g-minus-one}. Then one has
\begin{align*}
&\eta = \left[g,\prod_{\alpha\in\Sigma}x_\alpha(\zeta_\alpha)\right],\text{ where } (\zeta_\alpha)={}^{(g-1)^{-1}}(\xi_\alpha). \\
&g\cdot\prod_{\alpha\in\Sigma}x_\alpha(\zeta_\alpha)\cdot g^{-1} = \prod_{\alpha\in\Sigma}x_\alpha(\zeta'_\alpha)\text{ with } (\zeta'_\alpha)={}^g(\zeta_\alpha), \\
&\eta = \prod_{\alpha\in\Sigma}x_\alpha(\zeta''_\alpha)\text{ with }(\zeta''_\alpha)=(\zeta'_\alpha)-(\zeta_\alpha)={}^{(g-1)}(\zeta_\alpha)=(\xi_\alpha).
\end{align*}

Thus in this case $\theta=\eta$, a commutator.

\begin{figure}[htb]
\includegraphics{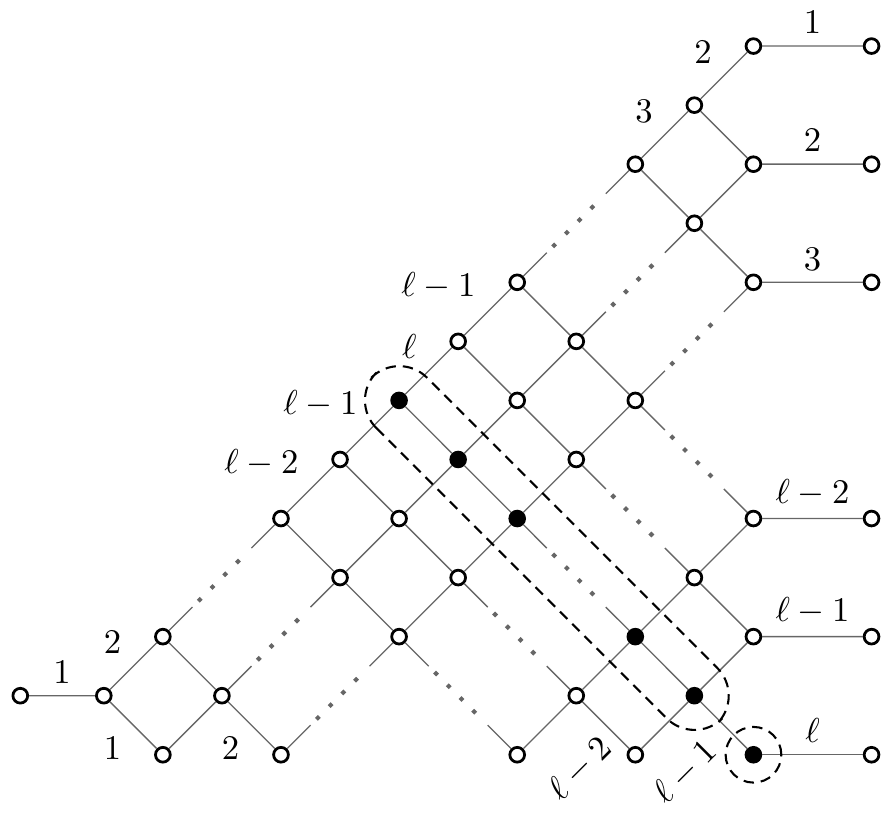}
\caption{$(\mathsf{C}_\ell,2\varpi_1)$}
\label{fig:cl-adj}
\end{figure}
\begin{figure}[htb]
\includegraphics{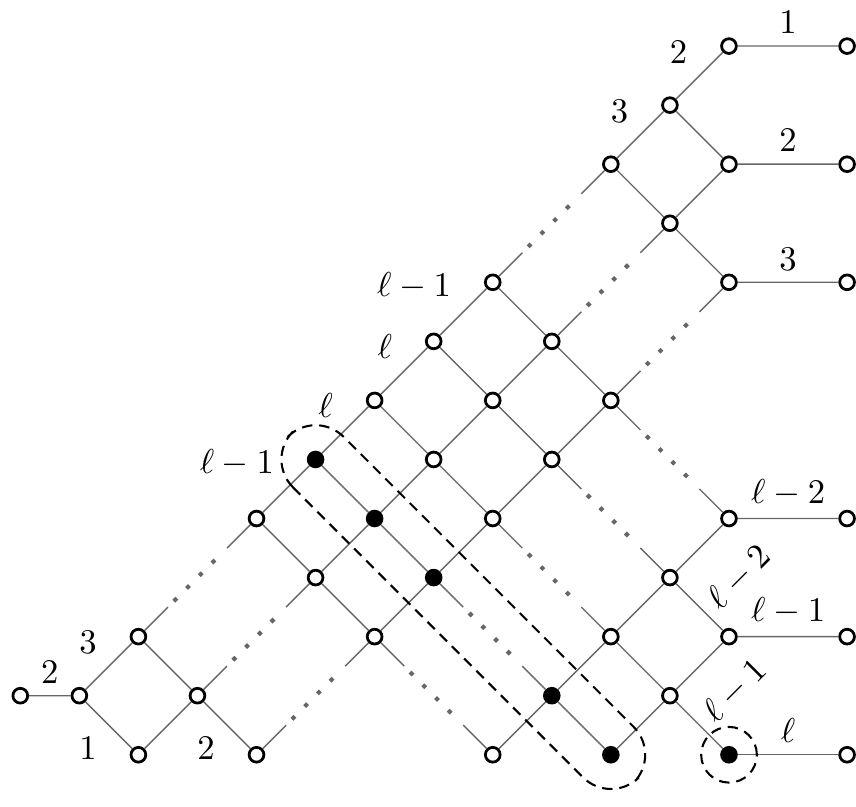}
\caption{$(\mathsf{B}_\ell,\varpi_2)$}
\label{fig:bl-adj}
\end{figure}
\begin{figure}[htb]
\includegraphics{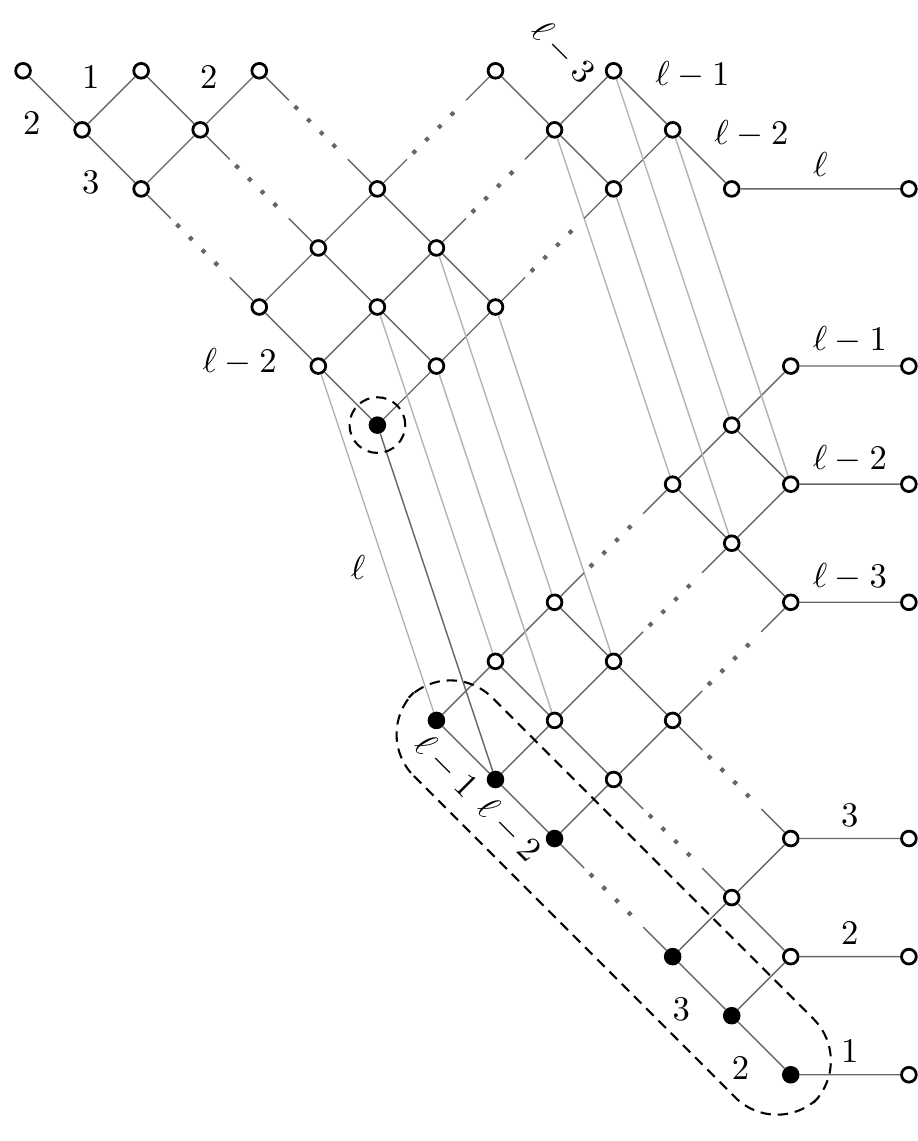}
\caption{$(\mathsf{D}_\ell,\varpi_2)$}
\label{fig:dl-adj}
\end{figure}
\begin{figure}[htb]
\includegraphics{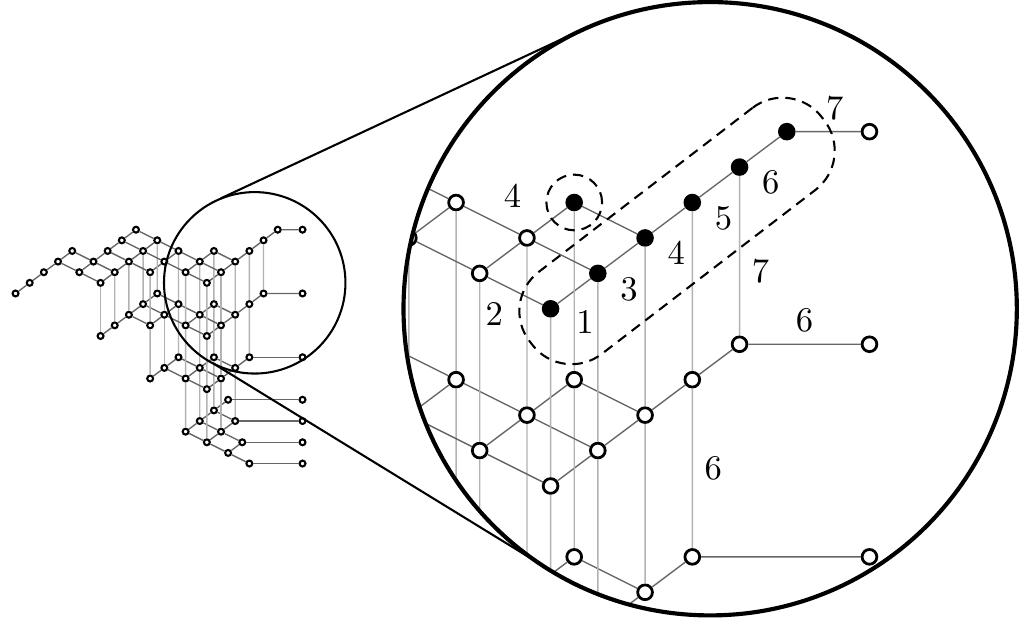}
\caption{$(\mathsf{E}_7,\varpi_1)$}
\label{fig:e7-adj}
\end{figure}
\begin{figure}[htb]
\includegraphics{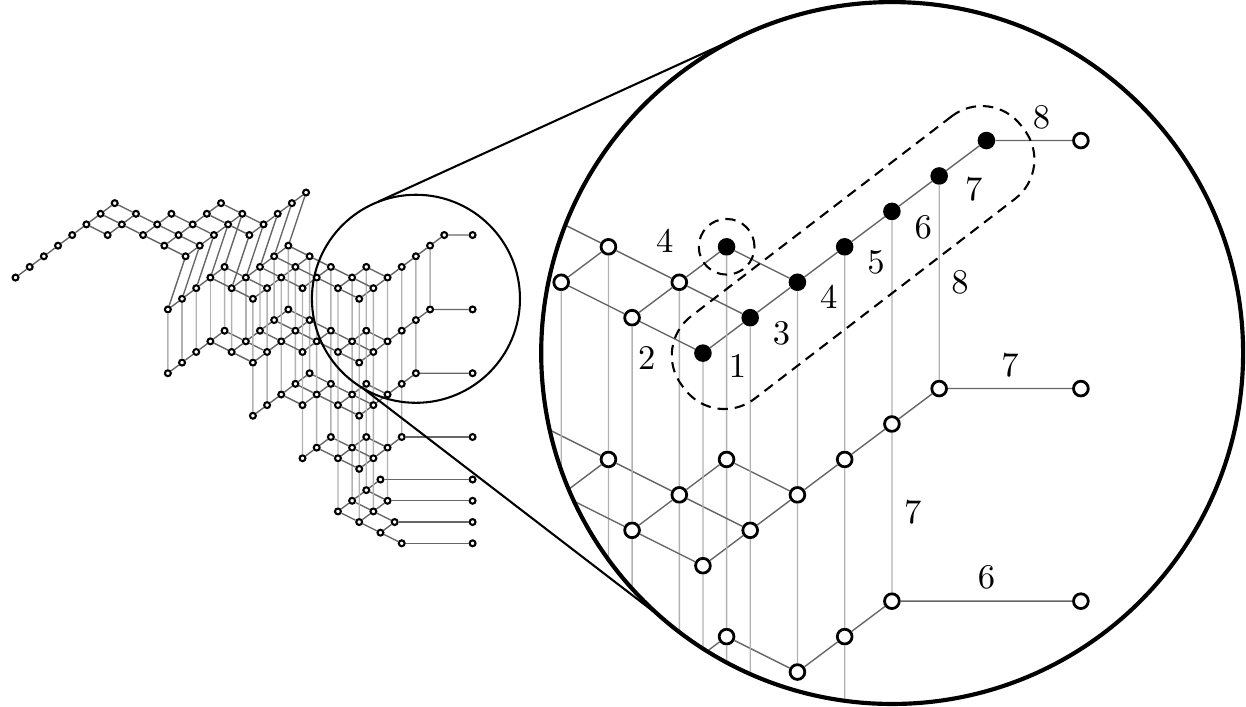}
\caption{$(\mathsf{E}_8,\varpi_8)$}
\label{fig:e8-adj}
\end{figure}
\begin{figure}[htb]
\includegraphics{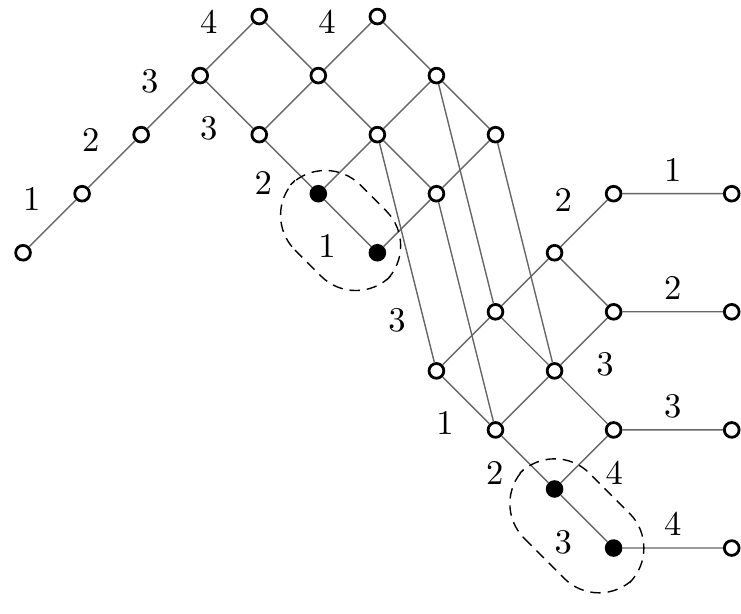}
\caption{$(\mathsf{F}_4,\varpi_1)$}
\label{fig:f4-adj}
\end{figure}
\begin{figure}[htb]
\includegraphics{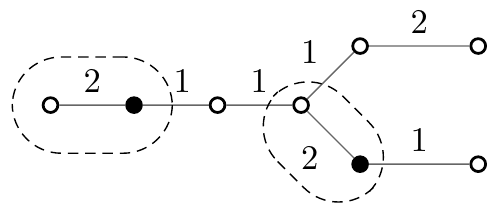}
\caption{$(\mathsf{G}_2,\varpi_2)$}
\label{fig:g2-adj}
\end{figure}
\begin{figure}[htb]
\includegraphics{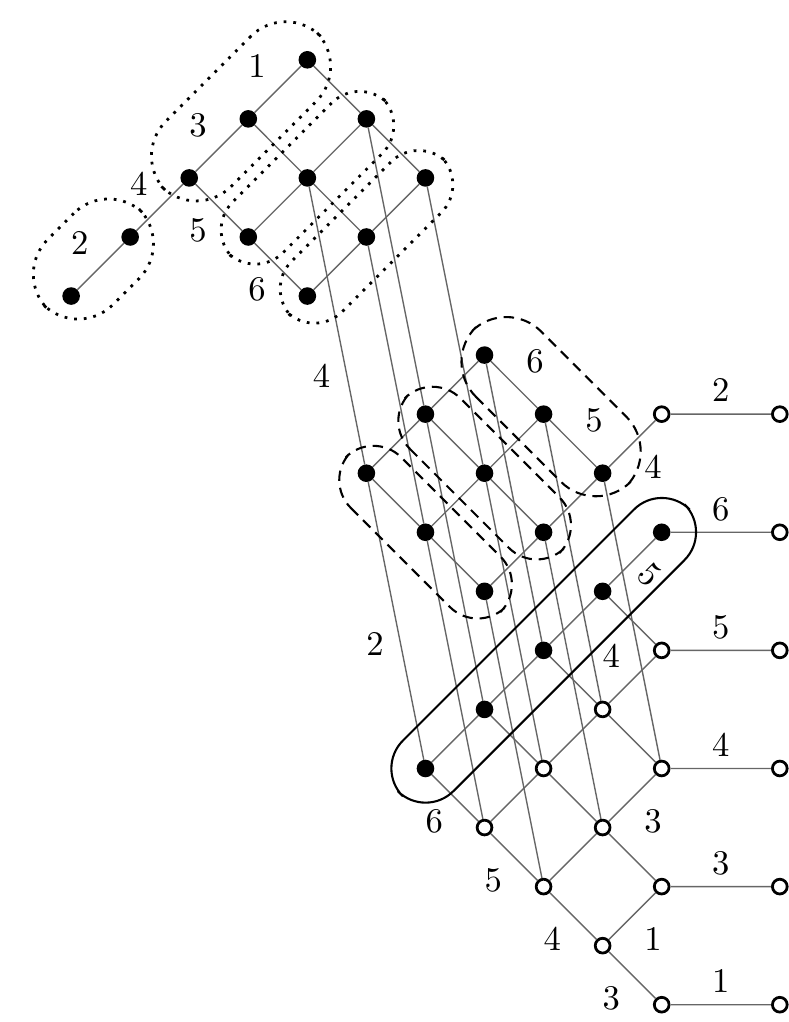}
\caption{$(\mathsf{E}_6,\varpi_2)$}
\label{fig:e6-adj}
\end{figure}

If $\Phi=\mathsf{C}_\ell$, we first write
\begin{align*}
& \theta=x_{\alpha_\ell}(a_\ell)\cdot\prod_{\alpha\in\Sigma'}x_\alpha(a_\alpha)=x_{\alpha_\ell}(a_\ell)\cdot\theta',\text{ where } \Sigma'=\Sigma\setminus\{\alpha_\ell\},\\
& x_{\alpha_{\ell-1}+\alpha_\ell}(*)\,x_{\alpha_\ell}(t)=\left[x_{2\alpha_{\ell-1}+\alpha_\ell}(1),x_{-\alpha_\ell}(\pm t)\right].
\end{align*}
Since $\alpha_{\ell-1}+\alpha_\ell\in\Sigma=\Sigma_\ell^{=1}\cap\Sigma_{\ell-1}^{\leqslant1}$,
\begin{align*}
& x_\alpha(t)=c\cdot x_{\alpha_{\ell-1}+\alpha_\ell}(*),\text{ where $c$ is a commutator,}\\
& \theta=c\cdot\theta'',\text{ for some } \theta''\in U(\Sigma\setminus\{\alpha_\ell\}).
\end{align*}
Then $g\in E(\Delta\cap\Delta_{\ell-1})$, constructed in Lemma \ref{lemma:g-minus-one}, doesn't use roots from $\Sigma_{\ell-1}$ and it is clear from the diagram (Figure~\ref{fig:cl-adj}) that it acts on $\theta''$ as prescribed by Lemma~\ref{lemma:g-minus-one}, allowing to repeat the argument we used for $\Phi=\mathsf{A}_\ell$. Thus $\theta$ is a product of two commutators.

If $\Phi=\mathsf{B}_\ell$ (Figure~\ref{fig:bl-adj}), we do the same as for $\mathsf{C}_\ell$. This time we exclude $\alpha_\ell$ and note that $\Delta_\ell\cap\Delta_{\ell-1}$ acts on the chain $\Sigma\setminus\{\alpha_\ell\}$ as $(\mathsf{A}_{\ell-2},\varpi_1)$. Again,
\[ x_{\alpha_\ell}(t)=\left[x_{\alpha_{\ell-1}+\alpha_\ell}(1),x_{-\alpha_{\ell-1}}(\pm t)\right]\cdot x_{\alpha_{\ell-1}+2\alpha_\ell}(*). \]

If $\Phi=\mathsf{D}_\ell$ (Figure~\ref{fig:dl-adj}), we divide $\Sigma$ into two parts: $\Sigma_1\cap\Delta_{\ell-1}$ and $\{\alpha=\alpha_1+\ldots+\alpha_{\ell-1}\}$. The first one is subject to the action of $\Delta_{\ell-1}$ (of type $\mathsf{A}_{\ell-1}$), while $x_\alpha(t)=[x_{\alpha+\alpha_\ell}(t),x_{-\alpha_\ell}(1)]$.

The very same method works for $\Phi=\mathsf{E}_7,\mathsf{E}_8$, see Figures~\ref{fig:e7-adj},~\ref{fig:e8-adj}.

In case $\Phi=\mathsf{F}_4$ (Figure~\ref{fig:f4-adj}) the subsystem subgroup $E(\langle\alpha_1,\alpha_3\rangle)\cong E(\mathsf{A}_1)\times E(\mathsf{A}_1)$ acts on $U(\Sigma)$, and the edges labeled $1$ and $3$ only meet each other outside $\Sigma$.

If $\Phi=\mathsf{G}_2$ (Figure~\ref{fig:g2-adj}), we slightly extend $\Sigma$ along the edges labeled $2$ to some $\Sigma'$, then $E(\langle\alpha_2\rangle)$ acts simultaneously on both chain of $U(\Sigma')$ (as a vector, it has zeroes on the additional roots).

In the remaining case $\Phi=\mathsf{E}_6$ (Figure~\ref{fig:e6-adj}) we split $\Sigma$ into three parts, marked by solid, dashed and dotted outlines in the figure, acted on by $E(\Delta_6)$, $E(\langle\alpha_5,\alpha_6\rangle)$ and $E(\langle\alpha_1,\alpha_2,\alpha_3\rangle)$ (of type $\mathsf{A}_5$, $\mathsf{A}_2$ and $\mathsf{A}_2\times\mathsf{A}_1$) correspondingly. 
\end{proof}
\begin{definition*}
A commutative ring $R$ is said to have \emph{stable rank $1$}, if for any $a,b\in R$, such that they generate $R$ as an ideal, there is $c\in R$, such that $a+bc\in R^*$ is invertible.
\end{definition*}
Examples of rings of stable rank $1$ are fields, semilocal rings, boolean rings, the ring of all algebraic integers, the disc-algebra.
\begin{theorem}\label{thm:comm-width-sr1}
Let $\Phi$ be a root system, $R$ a commutative ring of stable rank $1$. Then each element $g\in E(\Phi,R)$ is a product of at most $N$ commutators in $E(\Phi,R)$, where
\begin{itemize}
\item $N=3$ in case $\Phi=\mathsf{A}_\ell,\mathsf{F}_4,\mathsf{G}_2$;
\item $N=4$ in case $\Phi=\mathsf{B}_\ell,\mathsf{C}_\ell,\mathsf{D}_\ell,\mathsf{E}_7,\mathsf{E}_8$;
\item $N=5$ in case $\Phi=\mathsf{E}_6$.
\end{itemize}
\end{theorem}
\begin{proof}
We will use the so-called unitriangular factorization
\[ E(\Phi,R)=U^+(\Phi,R)\ U^-(\Phi,R)\ U^+(\Phi,R)\ U^-(\Phi,R), \]
that holds for any elementary Chevalley group over any commutative ring of stable rank $1$ \cite{VavSmSuUnitrEng}.

Assume first $\Phi\neq\mathsf{A}_{4k+1}$. Write $g=u_1v_1u_2v_2$, where $u_i\in U^+$ and $v_i\in U^-$. Then $g=u_3c_1v_3=c_2u_3v_3=c_2(u_3\pi)(\pi^{-1}v_3)$, where $c_i$ are commutators. Denote $\varphi=u_3\pi$ and $\psi=\pi^{-1}v_3$. By Lemmas~\ref{lemma:conjugate-to-companion} and \ref{lemma:conjugate-to-companion-minus} there exist $\mu,\nu\in E(\Phi)$ such that $z_1=\mu\varphi\mu^{-1}$ and $z_2=\nu\psi^{-1}\nu^{-1}$ are companion matrices. Since $z_1z_2^{-1}=\zeta\in U(\Sigma)$, one has
\[\mu\varphi\mu^{-1}=\zeta\nu\psi^{-1}\nu^{-1}.\]
Thus $\varphi=\mu^{-1}\zeta\nu\psi^{-1}\nu^{-1}\mu$ and
\begin{multline*}
\varphi\psi = \mu^{-1}\zeta\nu\psi^{-1}\nu^{-1}\mu\cdot\psi =\\=
\mu^{-1}\zeta\nu\cdot\psi^{-1}\cdot\nu^{-1}\cdot\left(\zeta^{-1}\mu\cdot\psi\cdot\psi^{-1}\mu^{-1}\zeta\right)\cdot\mu\psi =\\= \left[\mu^{-1}\zeta\nu,\psi^{-1}\right]\cdot\psi^{-1}\mu^{-1}\zeta\mu\psi = \left[\mu^{-1}\zeta\nu,\psi^{-1}\right]\cdot\zeta^{\mu\psi}.
\end{multline*}
Since $\zeta\in U(\Sigma)$ is a product of $N-2$ commutators by Lemma~\ref{lemma:companion-as-commutators}, we are done.

If $\Phi=\mathsf{A}_{4k+1}$, we start by writing $g=-u_1v_1u_2v_2$ for some $u_i\in U^+$, $v_i\in U^-$. Then, as previously, $g=-c_2\varphi\psi$, where $\varphi$ is similar to a companion matrix by Lemma~\ref{lemma:conjugate-to-companion}, while $\psi^{-1}$ is similar to a minus companion matrix by Remark~\ref{rmk:conjugate-to-companion-minus-a4k+1}. Then for $z_1=\mu\varphi\mu^{-1}$ and $z_2=\nu\psi^{-1}\nu^{-1}$ one has $z_1z_2^{-1}=-\zeta$ for some $\zeta\in U(\Sigma)$. Again,
\[ g=-c_2\varphi\psi=-c_2\cdot\left(-\left[\mu^{-1}\zeta\nu,\psi^{-1}\right]\cdot\zeta^{\mu\psi}\right) \]
is a product of $N$ commutators.
\end{proof}
\section{Final remarks}
We first note that starting with a uniriangular factorization of different length, one immediately obtains nice bounds on the commutator width. For example, Chevalley groups over boolean rings admit the unitriangular factorization $E(\Phi)=U^+U^-U^+$ of length $3$, thus any of its elements is conjugated to the product $uv$ for some $u\in U^+$, $v\in U^-$. It follow that each element of $E(\Phi,R)$ can be expressed as a product of $N-1$ commutators ($N$ as in Theorem~\ref{thm:comm-width-sr1}).

Another example is $E(\Phi,\mathbb{Z}[1/p])$, which admits the factorization of length $5$ \cite{VavSmSuUnitrEng,VseUnitrZ1p}, thus having the same estimate on its commutator width, as groups over rings of stable rank $1$.

In \cite{VasWheCommComp,ArlVasYouCommUnitary} the commutator width is computed also for what is called the \emph{extended} classical groups (the examples being $GL_n$, $GSp_{2n}$, $GO_n$, etc.). The resulting estimates are slightly better, because one can start with the Gauss decomposition \cite{VavSmSuGauss}
\[ E(\Phi,R)=H(\Phi,R)U^+(\Phi,R)U^-(\Phi,R)U^+(\Phi,R) \]
instead of the unitriangular factorization. Then one modifies Lemma~\ref{lemma:conjugate-to-companion} as follows (here $\overline{T}_{\mathrm{sc}}(\Phi)$ is the extended torus, see \cite{BerMooExtendsions,VavWeightElements}):
\begin{lemma}\label{lemma:conjugate-to-companion-borel}
For any $b\in H(\Phi)U^+(\Phi)$ exists $\eta\in \overline{T}_{\mathrm{sc}}(\Phi)U^+(\Phi)$ such that $\eta b\pi\eta^{-1}$ is a companion matrix.
\end{lemma}
The result is obtained in the same way as in Theorem~\ref{thm:comm-width-sr1}. Moreover, in this setting there is no need to treat the case $\Phi=\mathsf{A}_{4k+1}$ individually, since one can put $\widehat{w_0}=p_n$ as in \cite{VasWheCommComp}.

One more interesting thing about \cite{ArlVasYouCommUnitary} is an even better estimate in case of even-dimensional orthogonal group $O_{2n}$. The trick is to use not the Coxeter element of the Weyl group $W(\mathsf{D}_\ell)$, but rather a certain Coxeter element of the $\mathsf{A}_{\ell-1}$ subsystem $\Delta_\ell$, composed with an inner automorphism, corresponding to the symmetry of $\mathsf{D}_\ell$ Dynkin diagramm. This allows to take $\Sigma=\Sigma_\ell\cap\Delta_{\ell-1}$, acted on by $E(\Delta_\ell\cap\Delta_{\ell-1})$, which is exactly the case of $\mathsf{A}_{\ell-1}$. However, this automorphism is inner only for $O_{2n}$, but not for $SO_{2n}$, despite what is claimed in \cite{ArlVasYouCommUnitary}.

The followng argument, showing that this automorphism is inner in $O_{2n}$, is due to S. Garibaldi.

Let $\rho:G\to GL(V)$ be an irreducible representation of $G$ with the hidhest weight $\lambda$. Multiplying the given automorphism $\sigma$ of $\Phi$ by an element of the Weyl group, we can assume that $\sigma(\Pi)=\Pi$ (and $\sigma$ sends dominant weights to dominant weights). We wish to find an $x\in GL(V)$ such that $\sigma(g)=xgx^{-1}$ for every $g\in G$. Proposition~2.2 of \cite{BerGarLarPreservers} says that such $x$ exists if and only if $\sigma(\lambda)=\lambda$. Now let $\rho$ be the natural representation of $O_{2n}$, and since $\varpi_1$ is fixed by the symmetry, such $x$ exists in $GL_{2n}$. For every $G$-invariant polynomial function $f$ on $V$, $xf$ is $\rho(G)$-invariant. But $\rho(G)=G$, so $xf$ is $G$-invariant. If $f$ is a nondegenerate quadratic form on $V$, $xf$ is $SO(f)$-invariant and is a nondegenerate quadratic form, so it must be a scalar multiple of $f$.

For even-dimensional Spin group no such element exist in Pin or Clifford group, as it must swap the highest weights of two half-spin summand of its spin representation.

\bibliographystyle{amsalpha}
\bibliography{comm-width}

\end{document}